\documentclass[a4paper,12pt]{article}

\usepackage[bookmarks,pagebackref]{hyperref}
\usepackage{amsrefs}
\usepackage[utf8]{inputenc}
\usepackage[english]{babel}
\usepackage{amssymb}
\usepackage{amsmath}
\usepackage{latexsym}
\usepackage{amsthm}
\usepackage{thmtools,thm-restate}
\usepackage{eucal}
\usepackage{bbm}
\usepackage{chngcntr}
\usepackage{apptools}
\usepackage{mathtools}
\usepackage{authblk}
\usepackage[pdflatex]{crop}
\usepackage[dvipsnames]{xcolor}
\usepackage{tikz}
\usepackage{tikz-cd}
\usepackage{microtype}
\usepackage{enumitem}
\usepackage{graphicx}
\usepackage{mathrsfs}
\usepackage[colorinlistoftodos]{todonotes}
\usepackage{yfonts}
\usepackage{tabularx}
\usepackage{color}
\usepackage{enumitem}

\allowdisplaybreaks

\newcommand{\auth}[0]{{Renan Assimos}}
\newcommand{\tit}[0]{{On the intersection of minimal hypersurfaces of $S^k$}}
\newcommand{\kw}[0]{{Minimal surfaces, maximum principle, Frankel's theorem, two-piece property, first eigenvalue of minimal surfaces, Yau's conjecture}}

\hypersetup{
pdfauthor={\auth},%
pdftitle={\tit},%
colorlinks, linktocpage=true, pdfstartpage=1, pdfstartview=FitV,%
breaklinks=true, pdfpagemode=UseNone, pageanchor=true, pdfpagemode=UseOutlines,%
plainpages=false, bookmarksnumbered, bookmarksopen=true, bookmarksopenlevel=1,%
hypertexnames=true, pdfhighlight=/O,%
urlcolor=black, linkcolor=black, citecolor=black, 
}
\numberwithin{equation}{section}

\theoremstyle{plain}

\newtheorem*{theorem*}{Theorem}
\AtAppendix{\counterwithin{thm}{section}}

\newtheorem*{defn*}{Definition}

\theoremstyle{definition}

\DeclareMathOperator{\test}{span}
\DeclareMathOperator{\dist}{dist}

\newcommand{\Z}{\mathbb{Z}}

\newcommand{\R}{\mathbb{R}}

\let\originalleft\left
\let\originalright\right
\renewcommand{\left}{\mathopen{}\mathclose\bgroup\originalleft}
\renewcommand{\right}{\aftergroup\egroup\originalright}

\title{\tit}
\author{\auth\thanks{Correspondence: \href{mailto:assimos@mis.mpg.de}{assimos@mis.mpg.de}}}
\affil{\small Max Planck Institute for Mathematics in the Sciences\\ Leipzig, Germany}
\date{}

\begin{document}

\maketitle

\begin{abstract}
 It is known since the work of Frankel \cite{frankel66}, that two compactly immersed minimal hypersurfaces in a manifold with positive Ricci curvature must have an intersection point. Several generalizations of this result can be found in the literature, for example in the works of Lawson \cite{lawson70unknottedness}, Petersen and Wilhelm \cite{petersen2003}, among others. In the special case of minimal hypersurfaces of $S^k$, we prove a stronger version of Frankel's theorem. Namely, we show that if two compact minimal hypersurfaces $M_1$, $M_2$ of $S^k$ and a point $\mathbf{p}\in S^k$ are given, then $M_1$ and $M_2$ have an intersection point in the hemisphere with respect to $\mathbf{p}$. As a corollary of this result, we give an alternative proof to Ros' two-piece property of minimal surfaces of $S^3$ \cite{ros1995two}, for the general dimension case. 
\end{abstract}


\section{Introduction}

In differential geometry, minimal surfaces are among the most studied objects. Their geometrical and analytical properties are particularly fascinating when they lie in spaces of constant curvature, such as the Euclidean space $\R^3$ and the sphere $S^3$. The present paper focus on the k-dimensional analog of the second case and the intersection properties of minimal hypersurfaces therein. Two very interesting surveys on minimal surfaces of $S^3$ are the ones by Brendle \cite{Brendle2013survey} and Choe \cite{choe2006}. We follow the first while introducing concepts with an addendum: Every manifold in the present paper is considered to be connected, complete and without boundary, unless otherwise stated.

Minimal hypersurfaces of spheres, and more generally of complete Riemannian manifolds with positive Ricci curvature, possess intersection properties. In 1966, one of the first and most remarkable results about intersections of immersed minimal hypersufaces was found by Frankel \cite{frankel66}.

\begin{theorem*}[Frankel's Theorem] 
	 Let $M_1^{k-1}$ and $M_2^{k-1}$ be immersed minimal hypersurfaces in a Riemannian manifold $N^k$ of positive Ricci curvature.  If $M_1$ is compact, then $M_1$ and $M_2$ must intersect.
\end{theorem*}

\noindent While studying the regularity of harmonic maps into spheres, Jost, Xin and Yang \cite{jost2012} have obtained the following theorem, which has many applications to Bernstein problems \cite{jost2012}, \cite{assimos2018}, \cite{ding2018bernstein}.

\begin{restatable}{lemma}{primelemma}(Adapted from \cite{jost2012}) \label{minimal intersects half equator}
	For any $\pmb{\nu}\in S^k$, take $\pmb{\eta}\in S^{k-1}[\pmb{\nu}]$ and define $S^{k-1}_{+,\pmb{\eta}}[\pmb{\nu}] := S^{k-1}[\pmb{\nu}]\cap \overline{H_{+}(\pmb{\eta})}$ (see Figure~\ref{Fig1}). Then any $M^{k-1}\subset S^k$ compact immersed minimal hypersurface must intersect $S^{k-1}_{+,\pmb{\eta}}[\pmb{\nu}]$.
\end{restatable}

\begin{figure}[ht]
	\begin{picture}(100,160)
	\put(120,0){\includegraphics[width=0.4\linewidth]{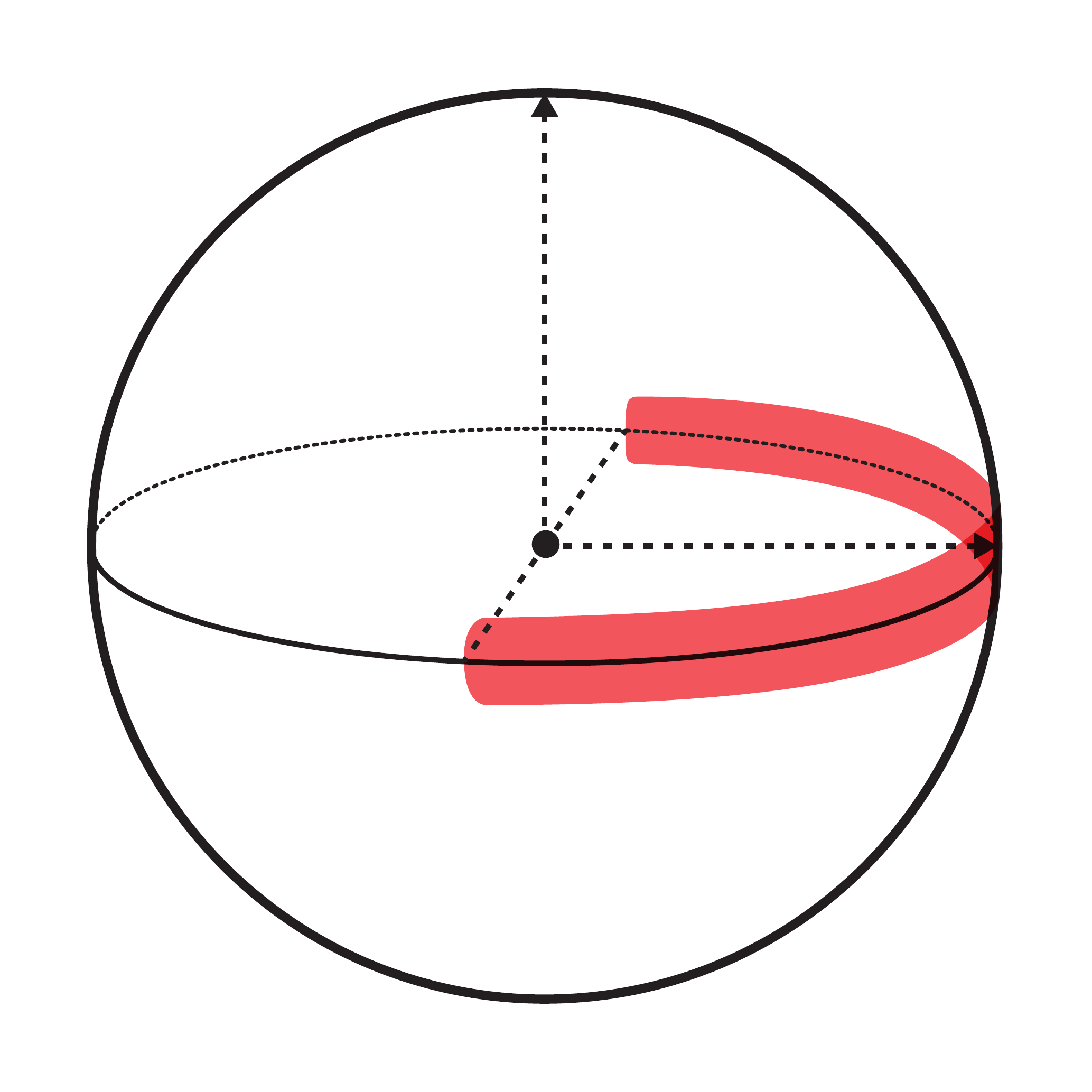}}
	\put(210,45){\color{red}$S^{k-1}_{+,\pmb{\eta}}[\pmb{\nu}]$}
	\put(148,49){$S^{k-1}[\pmb{\nu}]$}
	\put(197,148){$\pmb{\nu}$}
	\put(270,80){$\pmb{\eta}$}
	\end{picture}
	\caption{Half of the $(k-1)$-dimensional equator $S^{k-1}[v]$ in the positive direction of $\pmb{\eta}$, denoted by $S^{k-1}_{+,\pmb{\eta}}[\pmb{\nu}]$.}
	\label{Fig1}
\end{figure}
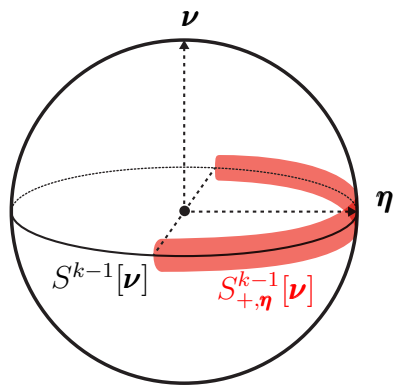

\noindent The precise meaning of the above notations are given in the next section, by now we build intuition looking at Figure~\ref{Fig1}. 

We are adapting the above result to our case of interest, which are the intersection between minimal hypersurfaces, but Theorem~\ref{minimal intersects half equator} remains true if we change $M$ by the image of any non-constant harmonic map defined on a closed Riemannian manifold. Nevertheless, looking at this theorem from `Frankel's perspective', if we have one of the minimal hypersurfaces being a totally geodesic $(k-1)$-equator, then we have that in the hemisphere with respect to any point of $S^k$, there is an intersection point between such $(k-1)$-equator and any given immersed minimal hypersurface. Using the same line of thought, the main theorem of the present paper can be seen as the generalization of Theorem~\ref{minimal intersects half equator} to the case where we consider intersections between an immersed minimal hypersurface and just `half' of the other. Let us state it as follows.

\begin{restatable}{lemma}{primelemmaa}\label{intersection of minimal hypersurfaces}
	Let $M_1^{k-1}$ and $M_2^{k-1}$ be immersed minimal hypersurfaces of $S^k$. For any $\pmb{\nu}\in S^k$, we have that\begin{equation*}
	M_1\cap M_2\cap \overline{H_+(\pmb{\nu})}\neq \emptyset,
	\end{equation*}
	where $\overline{H_+(\pmb{\nu})}$ is the closed hemisphere in the positive direction of $\pmb{\nu}$.
\end{restatable}

\noindent Therefore, in the special case of the sphere $S^k$, a stronger intersection property of compact immersed minimal hypersurfaces than the one of Frankel holds. Moreover, as we will see in the last section, an easy manipulation of Theorem~\ref{intersection of minimal hypersurfaces} gives a proof of Ros' two-piece property of compact embedded minimal hypersurfaces of spheres \cite{ros1995two}.

\begin{restatable}{lemma}{primelemmaaa}(Two-piece property) \label{twopiece property}
	Let $M^{k-1}$ be a connected, compact, embedded minimal hypersurface of $S^k$. For any $\pmb{\nu}\in S^k$, we have that $M_+(\pmb{\nu}):= M \cap H_{+}(\pmb{\nu})$ is connected. In other words, any $(k-1)$-dimensional equator devides $M$ in two connected pieces.
\end{restatable}   

\noindent Although this problem was originally proven in 1995 by Ros for the case of minimal surfaces in $S^3$ \cite{ros1995two}, it was long believed to be true, since it is a direct consequence of Yau's conjecture. More precisely, Yau has conjecture that the smallest positive eigenvalue of the Laplacian of $M$, $-\Delta_M$, is equal to $k-1$, provided that $M$ is an embedded minimal hypersurface of $S^k$. If we assume the conjecture to be true, then the two-piece property is an easy consequence of Courant's nodal theorem.  

Using the two-piece property and the fact that Lawson and Karcher-Pinkall-Sterling surfaces are invariant under some reflections in $S^3$,  Choe and Soret \cite{choe2009} have verified Yau's conjecture for each of these minimal surfaces of $S^3$. They also show that Ros' proof works in the general case of $S^k$, for $k\geq 3$. It is then worth to point out clearly that Theorem~\ref{twopiece property} in these notes consists of an alternative proof of Ros' result, while Theorem~\ref{intersection of minimal hypersurfaces} is new to the best of the author's knowledge.

Last, but not least, as an important tool in the proof of Theorem~\ref{intersection of minimal hypersurfaces}, we prove the lemma below, which is a very simple consequence of the maximum principle and, at the same time, interesting on itself due to its geometrical simplicity.

\begin{restatable}{lem}{primelem}\label{lemma1}
	Let $M^{k-1}_1$ and $M^{k-1}_2$ be two compact, immersed minimal hypersurfaces of $S^k$. For a point $\pmb{\nu} \in S^k$ and a number $r\leq \frac{\pi}{2}$, suppose that $M_i\cap B(\pmb{\nu}, r)\neq \emptyset$, for each $i \in \{1,2\}$, but
	\begin{equation*}
	M_1\cap M_2\cap \overline{B(\pmb{\nu}, r)} = \emptyset.
	\end{equation*}
	Moreover, let $\mathbf{a}\in M_1$ and $\mathbf{b}\in M_2$ be such that
	\begin{equation}\label{dist a and b}
	\dist_{\overline{B(\pmb{\nu}, r)}}(M_1, M_2) = \dist(\mathbf{a, b}).
	\end{equation}
	Then either $\mathbf{a}\in \partial B(\pmb{\nu}, r)$ or $\mathbf{b}\in \partial B(\pmb{\nu}, r)$.  
\end{restatable}

\noindent I would like to thank many people for their support and help during the preparation of this paper. I am in great debit with J\"urgen Jost for his constant support and all the discussions. I thank Slava Matveev for the discussions and the ideas for some of the pictures in the present paper. I thank Jingyong Zhu for all useful comments. I thank Luciano Mari for suggesting me some literature. Last, but not least, I thank Thomas Endler and Caio Teodoro for transforming my amateur hand drawn pictures into beautiful images.

\section{Main tools}

Let $M$ be a $(k-1)$-dimensional hypersurface in $S^k$, and let $\pmb{\nu}$ be a unit normal vector field along $M$. The extrinsic curvature of $M$ is described by the so-called second fundamental form of $M$. It is a symmetric two-tensor defined by 
\begin{equation*}
h(\mathbf{e_i},\mathbf{e_j}) = \langle D_{e_i}\pmb{\nu},\mathbf{e_j}\rangle,
\end{equation*}
where $\{e_i\}_{i=1}^{k-1}$ is an orthonormal basis of tangent vectors to $M$. The eigenvalues of such $h$ are called the principal curvatures of $M$ and the \textit{mean curvature} of $M$ is defined to be the sum of such eigenvalues
\begin{equation*}
H =\frac{1}{k-1} \left(\lambda_1 + ... + \lambda_{k-1}\right) =  \frac{1}{k-1}\sum_{i=1}^{k-1} \langle D_{e_i}\pmb{\nu},\mathbf{e_i}\rangle
\end{equation*}

\begin{defn*}
	A  hypersurface $M$ in $S^k$ is called minimal if its mean curvature vanishes identically.
\end{defn*}

\noindent The easiest example of such hypersurfaces are the $(k-1)$-dimensional equators, and they are going to play a significant role in the present paper: Let $\pmb{\nu}\in S^k$ be a point and define the following set.
\begin{equation}
S^{k-1}[\pmb{\nu}]=\{\mathbf{p} \in S^k\subset \R^{k+1}: \langle \mathbf{p}, \pmb{\nu}\rangle = 0\}.
\end{equation} 

\noindent The principal curvatures of this surface are all equal to zero.

When $k=3$, there are very classical and important examples of minimal surface in $S^3$, such as the Clifford torus, defined by the equation below.
\begin{equation*}
\Sigma = \{\mathbf{x}\in S^3 : x_{1}^2+x_{2}^2=x_{3}^2+x_{4}^2 =\frac{1}{2}\}. 
\end{equation*}
In this case, the principal curvatures are $1$ and $-1$, and therefore the mean curvature is zero.

Although a huge progress has been done in the past decade, not a great variety of explicit examples of minimal hypersurfaces in $S^k$ are known, specially when $k>3$. It is always worth to mention the most classical ones for the case of $k=3$: Lawson \cite{lawson70} constructed compact embedded minimal surfaces of arbitrary genus, and Karcher-Pinkall-Sterling \cite{karcher88} added some more similar ones. Both in $R^3$ and in $S^3$, the lack of many explicit examples has been a main obstacle to the study of embedded minimal surfaces. In 2014, Kapouleas \cite{kapouleas2017} has given the first examples of minimal surfaces in $S^3$ that are not symmetric under the antipodal map. Indeed, for each $l\geq 2$, he constructs minimal surfaces invariant under a $\Z_l$-rotation in $S^3$. Because of such symmetries, we claim that Choe and Soret's method \cite{choe2009} suffice to compute the first eingenvalues of the Laplacian for Kapouleas' surfaces. We suggest the reader to check the references in Kapouleas \cite{kapouleas2017}, as well as the work and references in Ketover \cite{ketover2016} for an overview of recent result.  

Analogously to the $(k-1)$-equators, we want to establish the following notation. Let $M^{k-1}$ be a compact, embedded, minimal hypersurface of the standard sphere $S^k$. For a given point $\pmb{\nu} \in S^k$ in addition to $S^{k-1}[\pmb{\nu}]$ defined above, we also consider the following hemispheres, with respect to $\pmb{\nu}$, and intersections.

\begin{equation}
H_{+}(\pmb{\nu}):= \{\mathbf{p} \in S^k: \langle \mathbf{p},\pmb{\nu} \rangle > 0 \},
\end{equation}
\begin{equation}
\overline{H_{+}(\pmb{\nu})}:= H_{+}(\pmb{\nu}) \cup S^{k-1}[\pmb{\nu}],
\end{equation}
\begin{equation}\label{def of M_+} 
M_{+}(\pmb{\nu}):= M\cap H_{+}(\pmb{\nu}),
\end{equation} 
\begin{equation}\label{def of M_+ bar}
\overline{M_{+}(\pmb{\nu})}:= M\cap  \overline{H_{+}(\pmb{\nu})}.
\end{equation}

\noindent In Equation~\eqref{def of M_+} and Equation~\eqref{def of M_+ bar}, we may omit $\pmb{\nu}$ if it is implicit that we are looking at the hemisphere of a fixed $\pmb{\nu}\in S^k$. We use the sigh `$-$' to denote the other hemisphere, \textit{e.g.}, $H_{-}(\pmb{\nu}):= \{\mathbf{p} \in S^k: \langle \mathbf{p},\pmb{\nu} \rangle < 0 \}$.

The strict maximum principle has several useful consequences and it will play the most fundamental role in the proof of every theorem herein. On the theory of minimal surfaces, the following version is particularly useful.

\begin{theorem*}[Maximum principle for minimal hypersurfaces]
Let $N$ be an $(n+ 1)$-dimensional smooth Riemannian manifold without boundary. Consider two connected minimal hypersurfaces $M_1^n$ and $M_2^n$ of $N$, at least one point $\mathbf{p}\in M_1\cap M_2$ and with the property that locally near each of their common points, one hypersurface lies on one side of the other. Then $M_1 = M_2$.
\end{theorem*}

\noindent In the remaining of this section, we give an alternative proof of Theorem~\ref{minimal intersects half equator}, with the following difference from the original one in Jost-Xin-Yang: Instead of arguing like the authors and constructing a Hildebrandt type strictly convex function on a given compact subset of the complement of $S^{k-1}_{+,\pmb{\eta}}[\pmb{\nu}]$, we apply a rotation argument, where we consider an isometry given by the rotation over a fixed plane, chosen so that the complement of $S^{k-1}_{+,\pmb{\eta}}[\pmb{\nu}]$ is foliated by a $1$-parameter family of $(k-1)$-equators. The reason we present this proof is because it builds intuition for the proof of Theorem~\ref{intersection of minimal hypersurfaces}.

\primelemma*

\begin{proof}
	Suppose Theorem~\ref{minimal intersects half equator} is false, that is, there exist $\pmb{\nu}, \pmb{\eta}\in S^k$ and a compact immersed minimal hypersurface $M^{k-1}$ of $S^k$ that does not intersect $S^{k-1}_{+,\pmb{\eta}}[\pmb{\nu}]$.
	
	Consider a rotation on the plane $\pi(\pmb{\nu}, \pmb{\eta})$, generated by $\pmb{\nu}$ and $\pmb{\eta}$, defined by 
	\begin{equation}\label{definition rotation}
	\begin{split}
	\rho_{\pmb{\nu}, \pmb{\eta}}(\theta, \cdot\,):S^k \to &S^k\\
	\mathbf{p} \mapsto & \left[ 
	\begin{array}{c@{}c@{}}
	\left[\begin{array}{cc}
	\cos(\theta) & -\sin(\theta) \\
	\sin(\theta) & \cos(\theta) \\
	\end{array}\right] &  \mathbf{0} \\
	\mathbf{0} & \left[\begin{array}{ccc}
	1 &  & \mathbf{0}\\ 
	 & \ddots & \\
	\mathbf{0} &  & 1\\
	\end{array}\right]\\
	\end{array}\right]\cdot \mathbf{p}
	\end{split}
	\end{equation}
	
	\noindent In the definition, $\mathbf{p}$ is seen as a point in $\R^{k+1}$ and written as follows. 
	\begin{equation*}
	\mathbf{p} = \mathbf{p}_\nu\cdot\pmb{\nu} + \mathbf{p}_\eta\cdot\pmb{\eta} + \sum_{i=1}^{k-1}\mathbf{p}_i\cdot \mathbf{e_i},
	\end{equation*}
	where $\left\{\mathbf{e_i}\right\}_{i=1}^{k-1}$ is an oriented orthonormal basis such that 
	\begin{equation*}
	\test\left\{\pmb{\nu}, \pmb{\eta}, \mathbf{e_1}, \cdots, \mathbf{e_{k-1}}\right\} = \R^{k+1}.
	\end{equation*}
	
	\noindent Let us define as well a $1$-parameter family generated by the rotation and $S^{k-1}_{+, \pmb{\eta}}[\pmb{\nu}]$ by the expression below.
	\begin{equation}
	S^{k-1}_{+, \pmb{\eta}}[\pmb{\nu}]\left(\theta\right) := \rho\left(\theta, S^{k-1}_{+, \pmb{\eta}}[\pmb{\nu}]\right).
	\end{equation}
	Since rotations are isometries of the sphere, for each $\theta\in S^1$, we have that $S^{k-1}_{+, \pmb{\eta}}[\pmb{\nu}]\left(\theta\right)$ is also totally geodesic. Moreover, the boundaries of each element of the above $1$-parameter family are the same. 
	\begin{equation}\label{boundary of half equator}
	\partial S^{k-1}_{+, \pmb{\eta}}[\pmb{\nu}]\left(\theta\right)= S^{k-2}[\pmb{\nu}, \pmb{\eta}].
	\end{equation}

	\noindent Since we are assuming that $M$ does not intersect $S^{k-1}_{+, \pmb{\eta}}[\pmb{\nu}]$ and $S^{k-1}_{+, \pmb{\eta}}[\pmb{\nu}]\left(\theta\right)$ foliates $S^k\setminus S^{k-2}[\pmb{\nu}, \pmb{\eta}]$, there exists $\theta^* > 0$ such that
	\begin{equation*}
	S^{k-1}_{+, \pmb{\eta}}[\pmb{\nu}]\left(\theta^*\right)\cap M \neq \emptyset,
	\end{equation*}
	and for a small $\epsilon > 0$
	\begin{equation*}
	S^{k-1}_{+, \pmb{\eta}}[\pmb{\nu}]\left(\theta^*-\epsilon\right)\cap M = \emptyset.
	\end{equation*}
	The intersection between these two hypersurfaces  must be an interior point of $S^{k-1}_{+, \pmb{\eta}}[\pmb{\nu}]\left(\theta^*\right)$, by Equation~\eqref{boundary of half equator}. But this contradicts the transversality of the intersection between minimal hypersurfaces given by the maximum principle.
\end{proof}

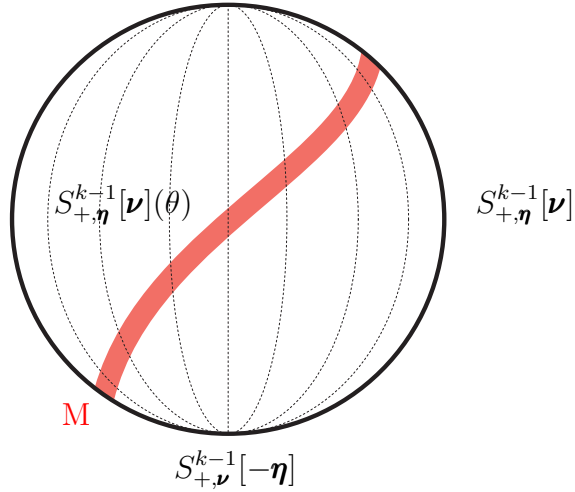
\begin{figure}	
	\begin{picture}(100,100)
	\put(120,0){\includegraphics[width=0.5\linewidth]{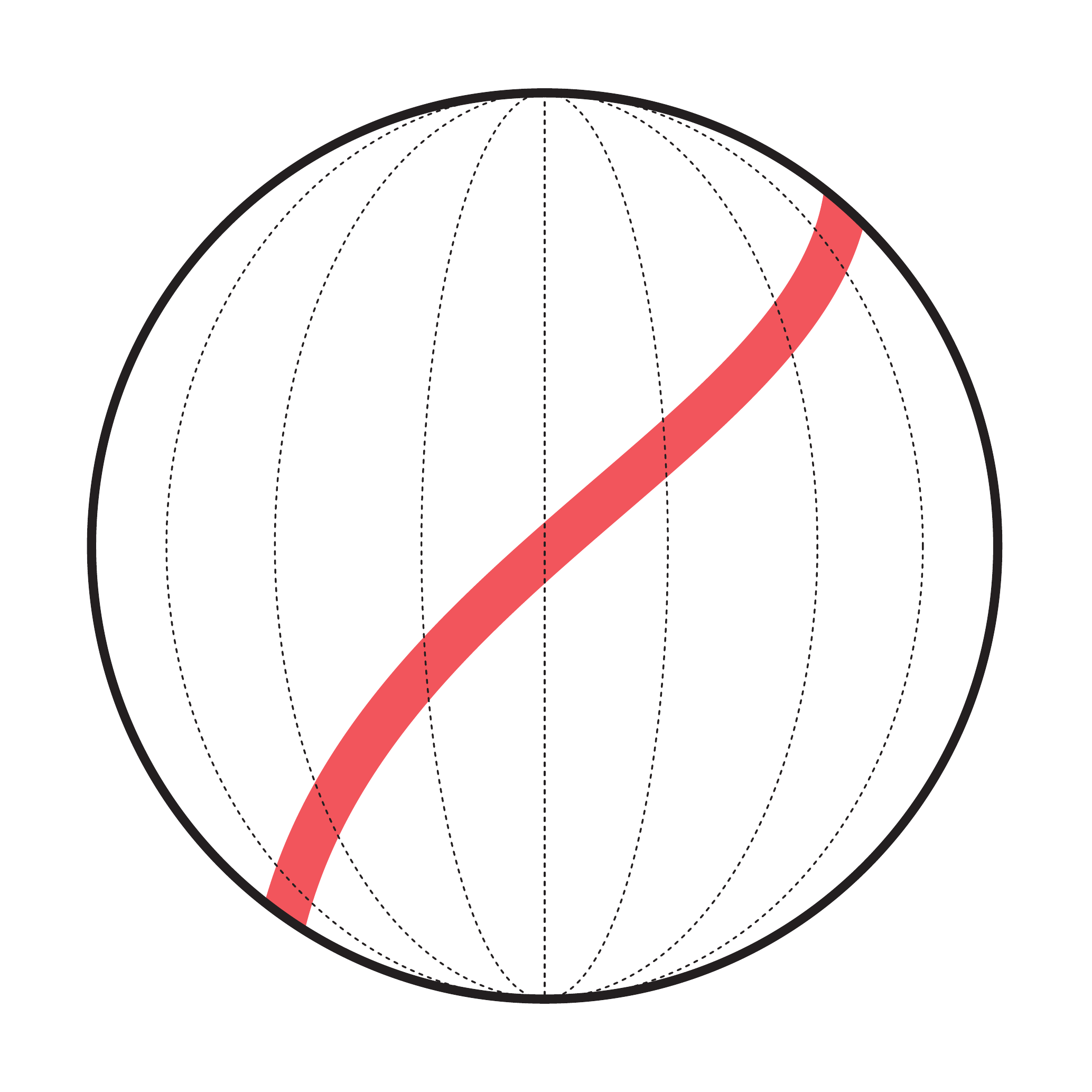}}
	\put(155,20){\color{red}M}
	\put(310,100){$S^{k-1}_{+, \pmb{\eta}}[\pmb{\nu}]$}
	\put(152,100){$S^{k-1}_{+, \pmb{\eta}}[\pmb{\nu}](\theta)$}
	\put(197,0){$S^{k-1}_{+, \pmb{\nu}}[-\pmb{\eta}]$}
	\end{picture}
	\caption{Projection of a hemisphere and a fixed foliation given by $(k-1)$-equators $S^{k-1}_{+, \pmb{\eta}}[\pmb{\nu}](\theta)$ (punctured lines) into the disk. A given minimal hypersurface $M$ (in red) must cross each of the leaves transversely and its boundary be in both sides of the central leaf $S^{k-1}_{+, \pmb{\nu}}[-\pmb{\eta}]$.}
	\label{Fig3}
\end{figure}

Theorem~\ref{minimal intersects half equator} itself is a very interesting result, and perhaps its direct applications wider than the ones of Theorem~\ref{intersection of minimal hypersurfaces}. For example, it indicates straightforwardly an `equilibrium' property of compact embedded minimal hypersurfaces of $S^k$ (immersed is unnecessarily general for what follows). 

Let us fix $\pmb{\nu}\in S^k$, and observe that for each point $\pmb{\eta}\in S^{k-1}[\pmb{\nu}]$, analogously to what we have done in the proof of Theorem~\ref{minimal intersects half equator}, we can define a foliation of the open half hemisphere $H_+(\pmb{\nu})$ using the $S^{k-1}_{+,\pmb{\eta}}[\pmb{\nu}]$ as leaves and vary $\pmb{\eta}$ in the direction of $\pmb{\nu}$ using the rotation defined by Equation~\eqref{definition rotation}. When $\rho_{\pmb{\nu}, \pmb{\eta}}(\theta,\pmb{\eta}) = \pmb{\nu}$, the corresponding leaf is dividing the $H_+(\pmb{\nu})$ in two equal parts.

Consider a compact embedded minimal hypersurface $M$ of $S^k$ and its intersection with $H_+(\pmb{\nu})$, denoted by $M_+$ (see Equation~\eqref{def of M_+}). On one hand, by the maximum principle, $M_+$ has to intersect transversaly each leaf of the foliation, and therefore must be in `both sides' of the central leaf when $\rho_{\pmb{\nu}, \pmb{\eta}}(\theta,\pmb{\eta}) = \pmb{\nu}$. On the other hand, the previous argument works regardless of which foliation by totally geodesic $(k-1)$-equators we are choosing, so we have points of $M_+$ in `both sides' of any foliation we may pick. Figure~\ref{Fig3} illustrates the projection to the disk of appropriate dimension of one of such foliations, and a minimal hypersurface $M$ going from one side of the boundary of $H_+(\pmb{\nu})$ to the other.

\section{Proof of the main results}

In this section we present the proves of Theorem~\ref{intersection of minimal hypersurfaces} and Theorem~\ref{twopiece property}. Since it will be important in the proof of the first theorem, we start by proving Lemma~\ref{lemma1}, which is a direct consequence of the maximum principle, but of independent interest, as we said in the introduction.

\primelem*

\begin{proof}
	If $M_1$ is a rotation of $M_2$, then their distance can be realized by points in the boundary of both hypersurfaces.
	
	Suppose that $M_1$ is not a rigid motion of $M_2$ and the lemma is false. That is, under the hypotheses for $M_1$, $M_2$, with $\mathbf{a}\in M_1$ and $\mathbf{b}\in M_2$ satisfying Equation~\eqref{dist a and b}, we have that $\mathbf{a}$ and $\mathbf{b}$ are interior points, i.e., $\mathbf{a}, \mathbf{b}\in B(\pmb{\nu}, r)$. See Figure~\ref{Fig4}.
	
	Consider the plane $\pi(\mathbf{a},\mathbf{b})$ generated by $\mathbf{a}$ and $\mathbf{b}$, and let $\{\mathbf{a}, \bar{\mathbf{b}}\}$ be the orthonormal basis given by Gram-Schmidt applied to $\{\mathbf{a}, \mathbf{b}\}$. Analogously to the proof of Theorem~\ref{minimal intersects half equator}, let us define the rotation $\rho_{\mathbf{a}, \bar{\mathbf{b}}}(\theta, \cdot)$ like in Equation~\eqref{definition rotation}.
	
	By the choice of the rotation, there exists $0 < \theta^* < \frac{\pi}{2}$ such that the hypersurface defined as
	\begin{equation*}
	M_1(\theta^*):= \rho_{\mathbf{a}, \bar{\mathbf{b}}}(\theta^*, M_1),
	\end{equation*}
	 intersects $M_2$ for the first time in $\overline{B(\pmb{\nu}, r)}$. Since $\mathbf{a}$ and $\mathbf{b}$ realize the distance between $M_1$ and $M_2$, the above intersection happens at the point $\mathbf{a}(\theta^*):=\rho_{\mathbf{a}, \bar{\mathbf{b}}}(\theta^*, \mathbf{a}) = \mathbf{b}$, an interior point. Therefore, $M_1(\theta^*)$ is minimal and lies on one side of $M_2$ in a neighborhood of $\mathbf{b}$, which implies that $M_1(\theta^*) = M_2$ by the maximum principle. But this contradicts our assumption that $M_1$ is not a rigid motion of $M_2$. 
\end{proof}

\begin{figure}
	\begin{picture}(150,150)
	\put(90,0){\includegraphics[width=0.5\linewidth]{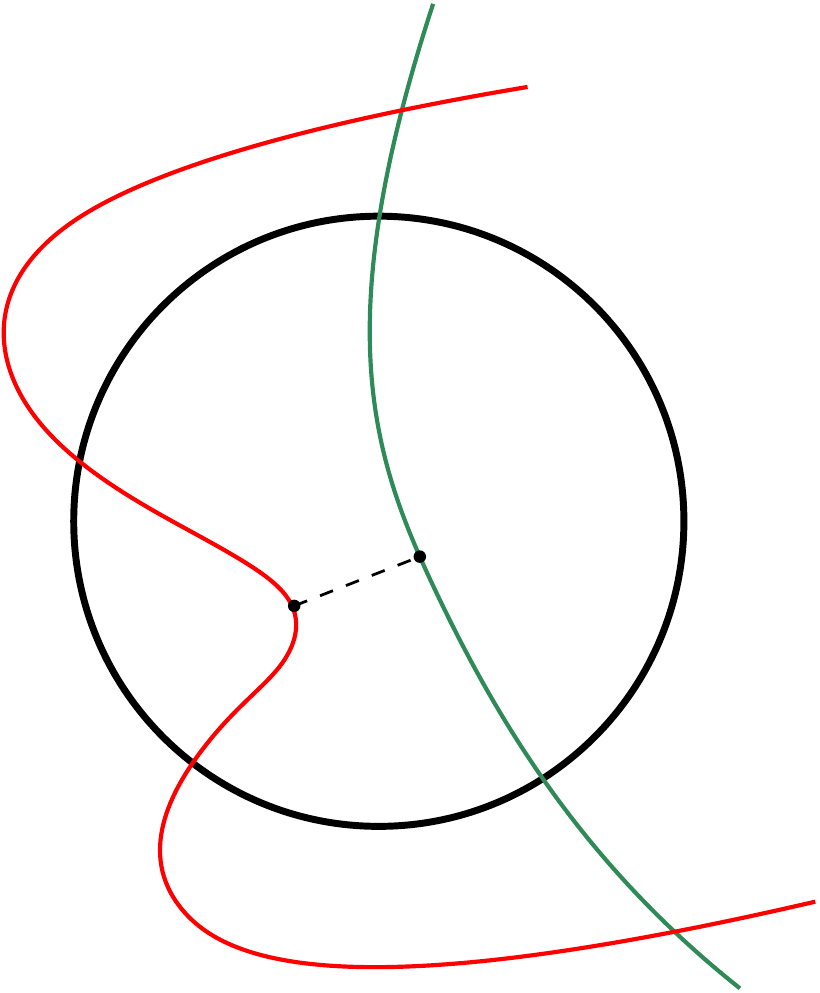}}
	\put(163,83){$\mathbf{a}$}
	\put(195,100){$\mathbf{b}$}
	\put(245,23){$\color{PineGreen!80!black}{M_2}$}
	\put(110,23){$\color{red!90!black}{M_1}$}
	\put(170,26){$B(\pmb{\nu}, r)$}
	\end{picture}
	\caption{The minimal hypersurfaces $M_1$ and $M_2$ of $S^k$ do not intersect inside the ball (they do intersect outside by Frankel's Theorem).}
	\label{Fig4}
\end{figure}
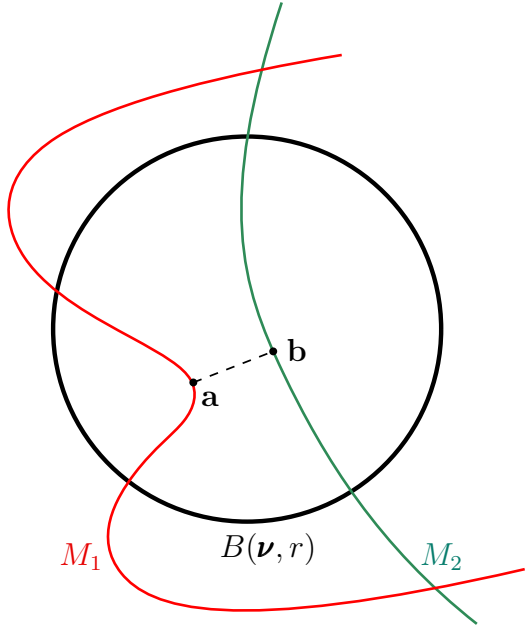

\primelemmaa*

\begin{proof}
	We start assuming that neither $M_1$ nor $M_2$ is totally geodesic, otherwise we are on the hypotheses of Theorem~\ref{minimal intersects half equator}. Moreover, if $M_1$ is a rigid motion of $M_2$, the result follows directly, so we assume that is not the case. 
	
	We argue by contradiction. Suppose there is no intersection between $M_1$ and $M_2$ in the hemisphere, that is, $M_1\cap M_2\cap \overline{H_{+}(\pmb{\nu})} = \emptyset$.
	
	By Lemma~\ref{lemma1}, without loss of generality, there exists $\mathbf{a} \in M_1\cap S^{k-1}[\pmb{\nu}]$ such that 
	\begin{equation*}
	\dist_{H_+(\pmb{\nu})}(\mathbf{a}, M^+_2) = \dist_{H_+(\pmb{\nu})}(M_1, M_2).
	\end{equation*}
	
	\noindent Analogously to the proof of Theorem~\ref{minimal intersects half equator}, we will define a rotation of $M_1$ that keeps $\mathbf{a}$ invariant and avoids the first intersection point to be a point at the boundary of $\overline{H_{+}(\pmb{\nu})}$. Actually, we are going to consider from now on just the connected component of $M_1$ in the hemisphere that contains $\mathbf{a}$ and regard it as being an embedding (the immersed case comes as corollary of this proof). We still denote it by $M_1$ to not introduce new notation. See Figure~\ref{Fig2}.
	
	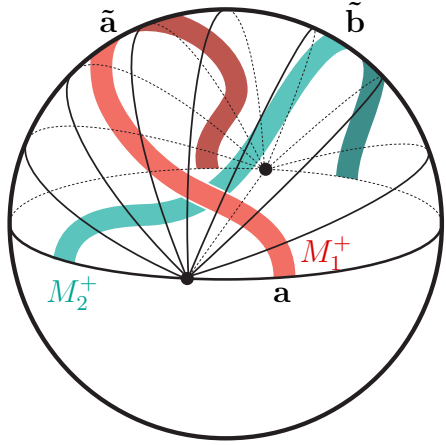
\begin{figure}
		\begin{picture}(0,150)
		\put(100,00){\includegraphics[width=0.5\linewidth]{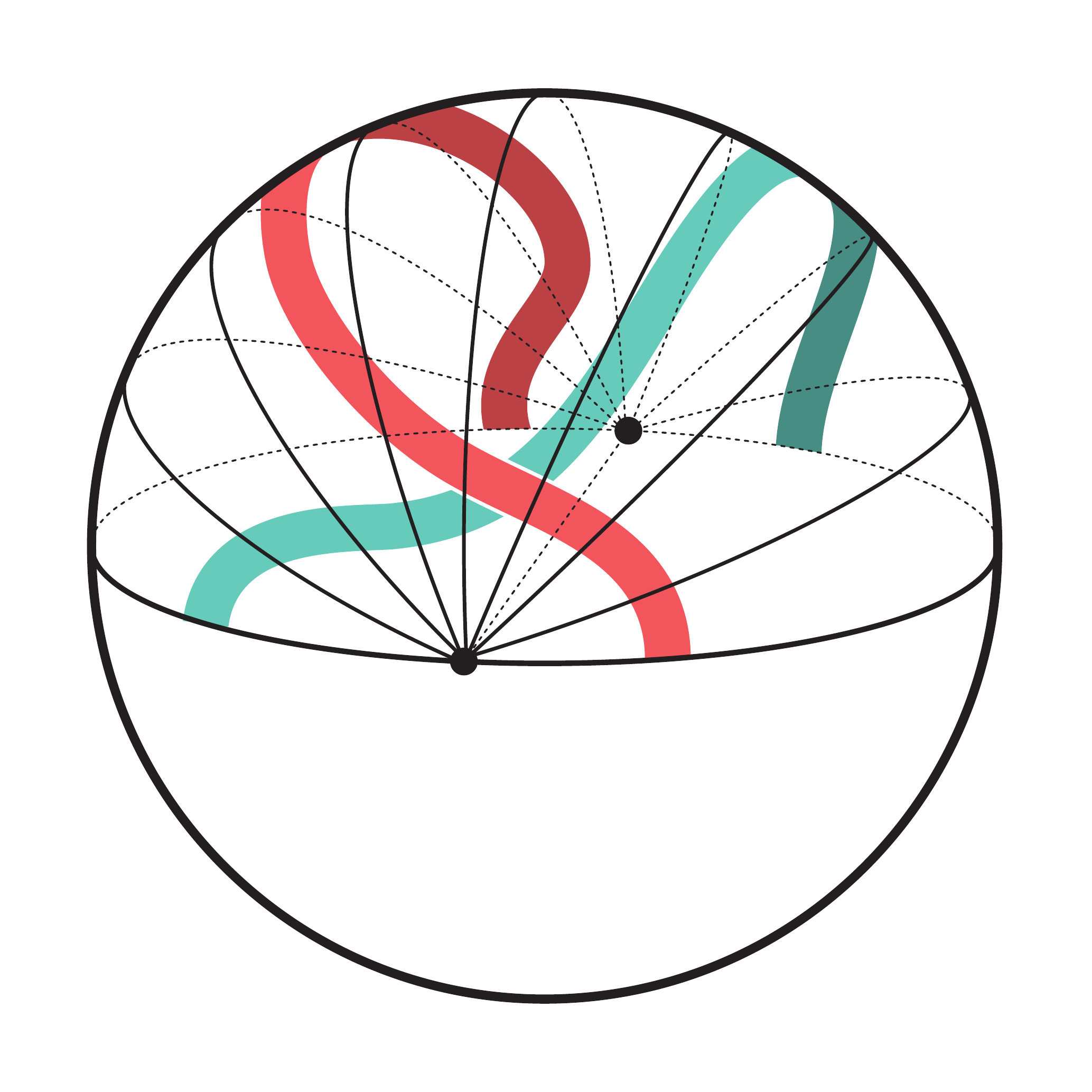}}
		\put(150,170){$\mathbf{\tilde{a}}$}
		\put(242,170){$\mathbf{\tilde{b}}$}
		\put(215,68){$\mathbf{a}$}
		\put(130,68){\color{JungleGreen!95!black}{$M_2^+$}}
		\put(225,83){\color{red!85!black}{$M_1^+$}}
		\end{picture}
		\caption{Take a foliation by $(k-1)$-equators of the hemisphere. Like in Figure~\ref{Fig3}, the minimal hypersurfaces $M_1$ and $M_2$ must intersect every leaf of the foliation transversaly, but we must imagine that, in very high dimensions, it is not obvious they do intersect.}
		\label{Fig2}
	\end{figure}
	
	Consider $S^{k-1}[\mathbf{a}]\cap\overline{H_{+}(\pmb{\nu})}$ and denote it by $S^{k-1}_{+, \mathbf{a}}[\pmb{\nu}]$. Since $M_1\cap S^{k-1}_{+, \mathbf{a}}[\pmb{\nu}] \neq \emptyset$ and $M_2\cap S^{k-1}_{+, \mathbf{a}}[\pmb{\nu}] \neq \emptyset$ by Theorem~\ref{minimal intersects half equator}, and both sets are compact, there exist points $\mathbf{\tilde{a}}\in M_1\cap S^{k-1}_{+, \mathbf{a}}[\pmb{\nu}]$ and $\mathbf{\tilde{b}}\in M_2\cap S^{k-1}_{+, \mathbf{a}}[\pmb{\nu}]$ such that they realize the distance between $M_1^+$ and $M_2^+$ inside $S^{k-1}_{+, \mathbf{a}}[\pmb{\nu}]$.
	
	Let $\pi(\mathbf{\tilde{a}}, \mathbf{\tilde{b}})$ be the plane generated by $\mathbf{\tilde{a}}$ and $\mathbf{\tilde{b}}$, and apply Gram-Schmidt to the latter, getting an orthonormal basis $\{\mathbf{\tilde{a}}, \mathbf{\bar{b}}\}$ for $\pi(\mathbf{\tilde{a}}, \mathbf{\tilde{b}})$. For $\theta \in [0,\pi]$, we define a rotation
	\begin{equation*}
	\rho_{\mathbf{\tilde{a}}, \mathbf{\bar{b}}}(\theta, \cdot): S^k \longrightarrow S^k
	\end{equation*} 
	on the above plane by Equation~\eqref{definition rotation}. This rotation keeps the $(k-2)$-equator $S^{k-2}[\mathbf{\tilde{a}}, \mathbf{\bar{b}}]$ invariant, and therefore $\mathbf{a}\in M_1(\theta)$, for every $\theta\in [0, \pi]$.
	
	Since we are considering just a connected component of the minimal hypersurface $M_1$, we write
	\begin{equation*}
	H_+(\pmb{\nu})\setminus M_1 := \Sigma \sqcup \tilde{\Sigma},
	\end{equation*} 
	where $\Sigma$ and $\tilde{\Sigma}$ are disjoint open sets with mean convex boundary. Since $M_2 \cap M_1 \cap H_+(\pmb{\nu}) = \emptyset$, suppose, without loss of generality, that $M_2\cap \Sigma = \emptyset$.
	
	Define a  $1$-parameter family of sets as follows.
	\begin{equation*}
	\Sigma(\theta) := \rho_{\mathbf{\tilde{a}}, \mathbf{\bar{b}}}(\theta, \Sigma).
	\end{equation*}
	Since $\rho_{\mathbf{\tilde{a}}, \mathbf{\bar{b}}}(\theta, \cdot)$ keeps $\mathbf{a}\in S^{k-1}[\pmb{\nu}]$ fixed, while moves $M_1$ to $M_2$ in the hemisphere $H_+(\pmb{\nu})$, and $\Sigma(\theta)$ has mean convex boundary for every $\theta$, the first point of intersection 
	\begin{equation*}
	\partial\Sigma(\theta^*) \cap M_2^+ \neq \emptyset,
	\end{equation*}
	between the $1$-parameter family and $M_2^+$, happens in a point $\mathbf{p}(\theta^*):=\rho_{\mathbf{\tilde{a}}, \mathbf{\bar{b}}}(\theta^*, \mathbf{p})\in\partial\Sigma(\theta^*)$ such that $\mathbf{p}\in M_1\cap H_+(\pmb{\nu})$. Since every $M_1(\theta)$ is a minimal hypersurface, the maximum principle implies that $M_1(\theta^*) = M_2$. Therefore $M_1$ is a rotation of $M_2$ and this contradicts our initial assumptions.

\end{proof}

\primelemmaaa*

\begin{proof}
	If $M$ is totally geodesic, the statement is trivial. So we suppose $M$ is a compact embedded minimal hypersurface of $S^k$, not an equator.
	Suppose there exists $\pmb{\nu}\in S^k$ such that
	\begin{equation*}
	M\cap H_+(\pmb{\nu}) = M_1 \sqcup M_2 \,\,\left(\,\sqcup \,\mathcal{R}\right),
	\end{equation*}
	where $M_1$ and $M_2$ are disjoint connected, compactly embedded minimal hypersurfaces in $\overline{H_{+}(\pmb{\nu})}$. By Theorem~\ref{intersection of minimal hypersurfaces}, if
	\begin{equation*}
	M_1\cap M_2 \cap H_+(\pmb{\nu}) = \emptyset,
	\end{equation*}
	then 
	\begin{equation*}
	M_1\cap M_2 \cap S^{k-1}[\pmb{\nu}] \neq \emptyset.
	\end{equation*}
	
	Since $M$ is an embedded minimal hypersurface, we have that $M_1$ and $M_2$ are actually the same geometric object in $H_+(\pmb{\nu})$. The fact that $M$ is compact guarantees that we need to repeat the above argument just a finite number of times to disregard possible disconnected components of $\mathcal{R}$.

\end{proof}

\bibliographystyle{alpha}
\bibliography{bernstein}
\addcontentsline{toc}{section}{\bibname}

\end{document}